\documentclass{article}
\usepackage[a4paper, total={6in, 9in}]{geometry}
\usepackage{amsfonts}
\usepackage{amsmath}
\usepackage{amssymb}
\usepackage{mathtools}
\usepackage{hyperref}
\usepackage{amsthm}

\newtheorem{theorem}{Theorem}
\newtheorem{lemma}[theorem]{Lemma}
\newtheorem{proposition}[theorem]{Proposition}

\newtheorem{corollary}[theorem]{Corollary}
\newtheorem{definition}[theorem]{Definition}
\newtheorem{conjecture}[theorem]{Conjecture}
\newtheorem{question}[theorem]{Question}
\newtheorem{remark}[theorem]{Remark}

\newcommand{\eps}{\varepsilon}
\newcommand{\D}{\mathcal{D}}
\newcommand{\CC}{\mathbb C}
\newcommand{\QQ}{\mathbb Q}
\newcommand{\ZZ}{\mathbb Z}
\DeclareMathOperator{\gen}{gen}
\newcommand{\lek}{safe}
\newcommand{\lekone}{almost-safe}
\newcommand{\paren}[1]{\left( #1 \right)}
\newcommand{\ceil}[1]{\left\lceil #1 \right\rceil}
\newcommand{\abs}[1]{\left\lvert #1 \right\rvert}
\newcommand{\set}[1]{\left\{ #1 \right\}}
\newcommand{\fromto}[2]{\colon #1 \to #2}

\numberwithin{theorem}{section}
\usepackage[backend=biber]{biblatex}
\title{Generic Classification and Asymptotic Enumeration of Dope Matrices}
\author{Ankit Bisain}
\date{}

\begin{document}

\pagebreak

\maketitle

\begin{abstract}
	For a complex polynomial $P$ of degree $n$ and an $m$-tuple of distinct complex numbers $\Lambda=(\lambda_1,\ldots,\lambda_m)$, the dope matrix $D_P(\Lambda)$ is defined as the $m \times (n+1)$ matrix $(c)_{ij}$ with $c_{ij} =1$ if $P^{(j)}(\lambda_i)=0$ and $c_{ij}=0$ otherwise. We classify the set of dope matrices when the entries of $\Lambda$ are algebraically independent, resolving a conjecture of Alon, Kravitz, and O'Bryant. We also provide asymptotic upper and lower bounds on the total number of $m \times (n+1)$ dope matrices. For $m$ much smaller than $n$, these bounds give an asymptotic estimate of the logarithm of the number of $m \times (n+1)$ dope matrices.
\end{abstract}

\section{Introduction}

Let $P \in \CC[x]$ be a polynomial of degree $n$, and let $\Lambda=(\lambda_1,\ldots,\lambda_m)$ be an $m$-tuple of distinct complex numbers. Following Alon, Kravitz, and O'Bryant \cite{dope}, we define the \emph{dope matrix of $P$ with respect to $\Lambda$} as the $m \times (n+1)$ matrix given by

\[D_P(\Lambda) \vcentcolon= \paren{c_{ij}}_{i \in [m], j \in [0,n]} \quad \text{where }c_{ij} = \begin{cases}1\ \text{if }P^{(j)}(\lambda_i)=0\\ 0\ \text{otherwise}. \end{cases}\]

Hence, the dope matrix tracks the pattern of common zeroes between $P$ and its derivatives -- that is, the set of ordered pairs $(i,j)$ for which we have $P^{(j)}(\lambda_i)=0$. A matrix is called \emph{dope} if it is of the form $D_P(\Lambda)$ for some $P$ and $\Lambda$. Denote by $\D_n^m$ the set of $m \times (n+1)$ dope matrices.

\

The set of possible dope matrices for fixed $\Lambda$ and $n$ may depend on the values of the $\lambda_i$. For example, if $P$ is a quadratic polynomial, then the conditions $P(\lambda_1)=0$, $P(\lambda_3)=0$, and $P'(\lambda_2)=0$ can be simultaneously satisfied only if $\lambda_2=\frac{\lambda_1+\lambda_3}{2}$. In the language of dope matrices, this is
\[D_P((\lambda_1,\lambda_2,\lambda_3)) = \begin{bmatrix}1 & 0 & 0\\ 0 & 1 & 0 \\ 1 & 0 & 0 \end{bmatrix} \implies \lambda_2=\frac{\lambda_1+\lambda_3}{2}.\]

For fixed $\Lambda$ and $n$, define
\[\D_n(\Lambda)=\{D_P(\Lambda) \mid P \in \CC[x],\ \deg P = n\}.\]

For any $a,b \in \CC$ with $b \neq 0$ and any $\QQ$-automorphism $\varphi$ of $\CC$, it can be shown (see \cite{multiplicity}) that
\begin{equation}\tag{*}
\D_n((\lambda_1,\ldots,\lambda_m))=\D_n((a+b\varphi(\lambda_1),\ldots,a+b\varphi(\lambda_m))). \label{eq:invariance}
\end{equation}

We denote the $m$-tuple on the right-hand side as
\[a+b\varphi(\Lambda) \vcentcolon= (a+b\varphi(\lambda_1),\ldots,a+b\varphi(\lambda_m)).\]

Define an \emph{affine algebraic dependence} of $\Lambda$ to be any rational-coefficient polynomial $P$ such that $P(a+b\lambda_1,a+b\lambda_2,\ldots,a+b\lambda_m)=0$ for all $a,b \in \CC$. For instance, $P(x_1,\ldots,x_m) \equiv 0$ is an affine algebraic dependence for any $\Lambda$, and $P(x_1,x_2,x_3)=x_2-\frac{x_1+x_3}{2}$ is an affine algebraic dependence of $(0,1,2)$. From now on, we refer to $P \equiv 0$ as the trivial affine algebraic dependence, and refer to all other affine algebraic dependences as nontrivial.

\

Call $\Lambda$ \emph{affinely algebraically independent} if it has no nontrivial affine algebraic dependences. Using (\ref{eq:invariance}), we can show (see Theorem \ref{thm:wellDefined}) that $\D_n(\Lambda)$ is the same for any affinely algebraically independent $m$-tuple $\Lambda$. We define $\D_n^{\gen(m)} \vcentcolon=  \D_n(\Lambda)$ for any choice of affinely algebraically independent $\Lambda$. (The notation $\D_n^{\gen(m)}$ is motivated by the fact that a generic $m$-tuple $\Lambda$ is affinely algebraically independent.) We remark that $\D_n^{\gen(m)}$ is natural to consider, as a generic $m$-tuple $\Lambda$ is affinely algebraically independent. In a related direction, Alon, Kravitz, and O'Bryant \cite[Theorem 6]{dope} have shown that $\abs{\D_n(\Lambda)}$ is maximized exactly when $\D_n(\Lambda)=\D_n^{\gen(m)}$.

\

In Section 2, we analyze $\D_n^{\gen(m)}$. Nathanson \cite[Theorem 2]{multiplicity} has characterized $\D_n^{\gen(m)}$ when $m=1$, and Alon, Kravitz, and O'Bryant \cite[Theorem 1]{dope} have characterized $\D_n^{\gen(m)}$ when $m=2$. We first generalize the aforementioned results to a complete characterization of $\D_n^{\gen(m)}$ for any $m$, resolving a conjecture of the latter paper \cite[Conjecture 15]{dope}.

\begin{theorem}\label{thm:genclass}
	For all positive integers $m,n$, the set $\D_n^{\gen(m)}$ consists of exactly the $m \times (n+1)$ matrices with $\{0,1\}$ entries such that for all $k \in [0,n]$, there are at most $k$ nonzero entries in the last $k+1$ columns.
\end{theorem}

Using this characterization, we are able to enumerate $\D_n^{\gen(m)}$.

\begin{theorem}\label{thm:enumeration}
	The number of elements of $\D_n^{\gen(m)}$ with $k$ ones is \[\frac{n+1-k}{n+1}\binom{(n+1)m}{k},\] and
	\[\binom{(n+1)m-1}{n}-(m-2)\sum_{k=0}^{n-1} \binom{(n+1)m-1}{k}\]
	is the size of $\D_n^{\gen(m)}$.
\end{theorem}

When $m=1$, this sum simplifies to $2^n$. When $m=2$, it simplifies to $\binom{2n+1}{n}$, matching a result of Alon, Kravitz, and O'Bryant \cite[Corollary 2]{dope}.

\

In Sections 3 and 4, we provide bounds on the size of $\D_n^m$, giving a partial answer to a question posed in \cite[Problem 16]{dope}.

\

In Section 3, we focus on the $m \leq \frac{n^2+n}{2}$ case. Alon, Kravitz, and O'Bryant \cite[Theorem 4]{dope} find an upper bound of $\binom{mn^2}{m+n}$ by applying a theorem of R\'{o}nyai, Babai, and Ganapathy \cite[Theorem 1.1]{zeroPatterns} regarding the number of zero-patterns of general sequences of polynomials. We improve this bound by directly applying the methods of \cite{zeroPatterns} to the question at hand. We also find a lower bound on $\abs{\D_n^m}$. For $m \leq \frac{n^2+n}{2}$, we have the following:

\begin{theorem}\label{thm:normalmbounds}
	For $m(t),n(t)\fromto{\ZZ_{>0}}{\ZZ_{>0}}$ satisfying $n(t) \to \infty$ and $1<m(t) \leq \frac{n(t)^2+n(t)}{2}$, we have \[(1+o(1))\log \paren{n^m\binom{mn}{n}} \leq \log \abs{\D_n^m} \leq (1+o(1))\log \paren{n^{2m}\binom{mn}{n}}.\]\label{thm:asymptotics}
\end{theorem}

Standard asymptotic notation, such as the $o(1)$ in the above theorem, is used throughout this paper and is explained in Section \ref{sec:prelims}. We also freely make statements regarding the asymptotic behavior of $\binom{mn}{n}$, which will follow from the asymptotic estimates of $\log \binom{mn}{n}$ in Section \ref{sec:prelims}.

\

The lower bound in Theorem \ref{thm:normalmbounds} comes from constructing a large set of elements of $\D_n^m$. One possible construction is to consider only the case when $\Lambda$ is generic, which gives a lower bound of $\log \abs{\D_n^{\gen(m)}}$.

\

Another construction comes from taking a ``generic" polynomial $P$. Consider a polynomial $P$ such that no two derivatives of $P$ have a common root and no derivative of $P$ has a root of multiplicity more than $1$. If we construct an element of $\D_n^m$ corresponding to the polynomial $P$ one row at a time, we have around $n$ choices for each row. Hence, this construction gives a lower bound of around $\log (n^m)$.

\

If $m=o(n)$, then only the $\log \binom{mn}{n}$ term of the lower and upper bounds matter asymptotically. In this regime, the bounds essentially match (see Theorem \ref{thm:correct} for a more precise statement) and $\D_n^{\gen(m)}$ accounts for the lower bound. When $m = \omega(n)$, only the $\log (n^m)$ term in the lower bound matters, and hence the generic $P$ construction accounts for the lower bound. The construction for the $m=\Theta(n)$ case is a combination of the generic $\Lambda$ construction and the generic $P$ construction.

\

In Section 4, we focus on the $m > \frac{n^2+n}{2}$ regime. As $P^{(j)}$ has at most $n-j$ roots for each $j$, any dope matrix can have at most $\frac{n^2+n}{2}$ nonzero rows. Hence, the growth rate of $\D_n^m$ in $m$ decreases after $m$ passes the threshold of $\frac{n^2+n}{2}$. Beyond this threshold, we have the following results:

\begin{theorem}\label{thm:largeasymptotics}
	For $m(t),n(t)\fromto{\ZZ_{>0}}{\ZZ_{>0}}$ satisfying $m(t),n(t) \to \infty$, we have
	\begin{enumerate}
		\item[(a)] If $\frac{n^2+n}{2}<m = \Theta(n^2)$, then $\log \abs{\D_n^m} \sim \log \abs{\D_n^{\frac{n^2+n}{2}}}$;
		\item[(b)] If $\log m = \omega(\log n)$, then $\log \abs{\D_n^m} \sim \log \paren{\dbinom{m}{\frac{n^2+n}{2}}}$.
	\end{enumerate}

	When $n$ is fixed, $\abs{\D_n^m}$ is a polynomial in $m$, and we can compute the two leading terms:
	\[\abs{\D_n^m} = \frac{\paren{\frac{n^2+n}{2}}!}{1!2!\cdots n!}\binom{m}{\frac{n^2+n}{2}}+\frac{\paren{\frac{n^2+n}{2}}!}{1!2!\cdots n!}\paren{1+\frac{(n-1)(n-2)}{4}}\binom{m}{\frac{n^2+n}{2}-1}+O\paren{m^{\frac{n^2+n}{2}-2}}.\]
\end{theorem}

\subsection{Preliminaries and Notation}\label{sec:prelims}

We now check that $\D_n^{\gen(m)}$ is well-defined.

\begin{lemma}\label{lem:shift}
	If an $m$-tuple $\Lambda=(\lambda_1,\ldots,\lambda_m)$ of distinct complex numbers is affinely algebraically independent, then there exist complex numbers $a,b$ such that $a+b\Lambda$ is algebraically independent.
\end{lemma}

\begin{proof}
	We have that for all polynomials $P \in \QQ[x_1,\ldots,x_m] \setminus \{0\}$, the polynomial $P(a+b\Lambda) \in \CC[a,b]$ is not identically zero. We want to show that there exist $a,b \in \CC$ such that $P(a+b \Lambda) \neq 0$ for all $P \in \QQ[x_1,\ldots,x_n] \setminus \{0\}$. We'll show, more generally, that if $S$ is any countable subset of $\CC[a,b] \setminus \{0\}$, then there exist $a,b$ such that $Q(a,b) \neq 0$ for all $Q \in S$.

	\

	View each $Q \in S$ as a polynomial in $b$ with coefficients being polynomials in $a$. For some $k$, the coefficient of $b^k$ in $Q$ is some nonzero polynomial $R$ in $a$. For any $a$ with $R(a) \neq 0$, we have that $Q(a,b)$ has finitely many roots as a polynomial in $b$. For each $Q \in S$, pick some such polynomial $R$, and let $A_Q$ be the set of roots of $R$. Since each $A_Q$ is finite, we have that $\bigcup_{Q \in S} A_Q$ is countable, and hence, we can pick some $a_0\in\CC$ that isn't in $A_Q$ for any $Q \in S$.
	
	\
	
	We claim that for some $b_0 \in \CC$, we have $Q(a_0,b_0) \neq 0$ for all $Q \in S$. Let $B_Q$ be the set of roots of $Q(a_0,b)$ viewed as a polynomial in $b$. By the definition of $a_0$, we have that each $B_Q$ is finite, so $\bigcup_{Q \in S}B_Q$ is countable. Hence, there is some $b_0 \in \CC$ that isn't in $B_Q$ for any $Q \in S$, as desired.
\end{proof}

We'll also recall a fact from \cite{multiplicity} allowing us to equate $D_n(\Lambda)$ and $\D_n(\Lambda')$ when we have linear maps or $\QQ$-automorphisms of $\CC$ sending $\Lambda$ to $\Lambda'$.

\begin{lemma}\cite[Theorem 5]{multiplicity}
	For any $\QQ$-automorphism $\varphi$ and complex numbers $a,b$ with $b \neq 0$, we have $\D_n(\Lambda)=\D_n(a+b\varphi(\Lambda))$.
\end{lemma}

The idea of the proof is, for any polynomial $P(x)=a_n x^n + a_{n-1}x^{n-1} + \cdots + a_0$, to consider the polynomial \[P_{a+b \varphi}(x) \vcentcolon= \varphi(a_n) \paren{\frac{x-a}{b}}^n+\varphi(a_{n-1})\paren{\frac{x-a}{b}}^{n-1}+\cdots+\varphi(a_0).\]
This polynomial has the property that $\D_P(\Lambda) = \D_{P_{a+b \varphi}}(a+b\Lambda)$ for all $P$. A more complete explanation can be found in \cite{multiplicity}.

\

We are now ready to check that $\D_n^{\gen(m)}$ is well-defined.

\begin{theorem}\label{thm:wellDefined}
	If $\Lambda_1$ and $\Lambda_2$ are affinely algebraically independent $m$-tuples of complex numbers, then $\D_n(\Lambda_1)=\D_n(\Lambda_2)$.
\end{theorem}

\begin{proof}
	Let $\Lambda$ be an algebraically independent $m$-tuple of complex numbers.
	By Lemma \ref{lem:shift}, there are constants $a_1,b_1,a_2,b_2 \in \CC$ such that $a_1+b_1\Lambda_1$ and $a_2+b_2\Lambda_2$ are algebraically independent. Now, there exists some $\QQ$-automorphism $\varphi$ of $\CC$ sending $a_1+b_1\Lambda_1$ to $a_2+b_2\Lambda_2$ (see the introduction of \cite{dope} for further discussion of this fact). Hence,
	\[\D_n(\Lambda_1) = \D_n(a_1+b_1 \Lambda_1) = \D_n(\varphi(a_1+b_1\Lambda_1)) = \D_n(a_2+b_n \Lambda_2) = \D_n(\Lambda_2),\]
	as desired.
\end{proof}

To state our results on asymptotic dope matrix counts, we will use standard asymptotic notation. For functions $f,g\fromto {\ZZ }{\mathbb R}$, we can compare the growth rates of $f$ and $g$ with the following:

\begin{enumerate}
	\item $f=O(g)$ if, for some constant $K$, we have $\abs{f} \leq K\abs{g}$ for all sufficiently large $t$;
	\item $f=\Theta(g)$ if, for some constants $K,K'$, we have $K\abs{g} \leq \abs{f} \leq K'\abs{g}$ for all sufficiently large $t$;
	\item $f \sim g$ if, as $t$ approaches $\infty$, we have that $\frac{f}{g}$ approaches $1$;
	\item $f=o(g)$ if, for any constant $\eps$, we have $\abs{f} \leq \eps\abs{g}$ for all sufficiently large $t$;
	\item $f=\omega(g)$ if, for any constant $M$, we have $\abs{f} \geq M\abs{g}$ for all sufficiently large $t$.
\end{enumerate}

We will frequently need to analyze the asymptotic behavior of $n!$ and $\log \binom{mn}{n}$. We will make use of the following inequalities:

\begin{enumerate}
	\item For any positive integer $n$, we have $n\log n - n \leq \log n! \leq n \log n$.
	\item For $m(t),n(t)\fromto{\ZZ_{>0}}{\ZZ_{>1}}$ satisfying $n(t) \to \infty$, we have:
		\[\log \binom{mn}{n} = n \log \paren{\frac{m^m}{(m-1)^{(m-1)}}}+O(\log n)=n\log m + O(n).\]
	In particular, $\log \binom{mn}{n}$ grows linearly when $m = O(1)$ and is $\omega(n)$ when $m = \omega(1)$.
\end{enumerate}

The inequality follows from
	\[n\log n - n \leq  \int_1^n \log x \mathrm{d} x \leq \log 2 + \log 3 + \cdots + \log n \leq n \log n,\]
since $\log n! = \log 2 + \log 3 + \cdots + \log n$. The first asymptotic estimate on $\log \binom{mn}{n}$ follows directly from Stirling's approximation, and the second follows from the fact that $2 \leq  \paren{\frac{m}{m-1}}^{m-1} < e$ for all $m >1$.
 
\section{Generic Dope Matrices}

In this section we will prove Theorem \ref{thm:genclass}. We first outline the proof of the $m=2$ case from \cite[Section 4]{dope}, as our proof uses many ideas from it. It makes use of the following result (slightly rephrased) of Gessel and Viennot.
\begin{theorem}\cite[Corollary 3]{binomdet}\label{thm:binomlinindep}
	Let $\mathcal{G},\mathcal{H} \subset \mathbb{Z}_{\geq 0}$ be finite sets. If $\abs{\mathcal{G} \cap [0,c]} \leq \abs{\mathcal{H} \cap [0,c]}$ for every $c \in \ZZ_{\geq 0}$, then the matrix of binomial coefficients
	\[\left[\binom{g}{h}\right]_{g \in \mathcal{G}, h \in \mathcal{H}}\]
	has rank $\abs{\mathcal{G}}$.
\end{theorem}

\begin{proof}[Outline of Proof of Theorem \ref{thm:genclass} when $m=2$]
	Take $\Lambda=(0,1)$, and let $P(x)=a_nx^n + a_{n-1}x^{n-1}+\cdots+a_0$. For a $2 \times (n+1)$ matrix $M$, we can view the equation $D_P(\Lambda)=M$ as a system of linear equations in $a_n,\ldots,a_0$ via
	\[P^{(s)}(0) = 0 \iff a_{s}= 0 \quad\text{and}\quad P^{(s)}(1)=0 \iff \sum_{i=0}^{n} \binom{n-i}{s}a_i=0.\]
	When $M$ does not satisfy the property in Theorem \ref{thm:genclass}, for some $k$, it has at least $k+1$ ones in the last $k+1$ columns. Take the minimum $k$ with the preceding property and look at the last $k+1$ columns. The resulting linear equations are linearly independent by Theorem \ref{thm:binomlinindep}.
	
	\
	
	In these linear equations, only the $k+1$ variables $a_n ,a_{n-1},\ldots,a_{n-k}$ have nonzero coefficients, so we must have $a_n=a_{n-1}=\cdots=a_{n-k}=0$. However, if $a_n=0$, then $P$ cannot be a degree-$n$ polynomial, which gives a contradiction.
	
	\
	
	To show that any matrix $M$ with at most $k$ ones in the last $k+1$ columns is attainable, we add all-one columns to the left end of the matrix until the number of ones is one smaller than the number of columns. Say we added $c$ columns, and let $M'$ be the matrix with the $c$ columns added. Take the resulting system of equations and append $a_{n+c}=1$, ensuring that our desired polynomial has degree exactly $n+c$.
	
	\
	
	By Theorem \ref{thm:binomlinindep}, the resulting equations are linearly independent, and hence have a unique solution. Letting $P_0$ be the polynomial corresponding to this solution, we have that $D_{P_0}(\Lambda)$ is not the zero matrix, since $P^{(n+c)}(0) \neq 0$. Hence, $D_{P_0}(\Lambda)$ must have at most $n+c$ ones, and is hence exactly $M'$. Now, $D_{P_0^{(c)}}(\Lambda)$ is $M$, as desired.
\end{proof}

We now introduce notation that will be helpful in our proof of Theorem \ref{thm:genclass}

\begin{definition}
	\normalfont Call a matrix \emph{\lek} if the last $k+1$ columns contain at most $k$ nonzero entries for all $k$. Similarly, call a matrix \emph{\lekone} if the last $k+1$ columns contain at most $k+1$ nonzero entries for all $k$.
\end{definition}

Throughout this section, we will let $P(x)=a_n x^n + a_{n-1}x^{n-1}+\cdots+a_{0}$ denote a general polynomial with $\deg P \leq n$. Note that the set of such $P$ forms a complex vector space of dimension $n+1$ with basis given by $1,x,\ldots,x^n$. Hence, for a fixed integer $s$ and complex number $t$, we can view $P^{(s)}(t)$ as a linear form in $P$:
\[(a_0,\ldots,a_n) \mapsto \sum_{j=0}^n {j(j-1)\cdots (j+1-s) t^{j-s}}a_j.\]

We will use $\Lambda=(\lambda_1,\ldots,\lambda_m)$ and $\mathbf{S}=(S_1,\ldots,S_m)$ to denote arbitrary $m$-tuples of complex numbers and $m$-tuples of subsets of $\{0,1,\ldots,n\}$, respectively. Denote by $P^{(\mathbf{S})}(\Lambda)$ the set of linear forms $P^{s}(\lambda_i)$ for $i \in [m]$ and $s \in S_i$, and denote by $M(\mathbf{S})$ the matrix $(c_{ij})_{i \in [m],\ j \in [0,n]}$, where $c_{ij}$ is $1$ if $j \in S_i$ and $0$ otherwise.

\subsection{Reduction to the Key Lemma}

The key lemma is the following linear independence property:

\begin{lemma}\label{lem:lekoneLinIndep}
	If $M(\mathbf{S})$ is \lekone\ and $\Lambda$ is generic, then $P^{\mathbf{S}}(\Lambda)$ is linearly independent.
\end{lemma}

Before proving this lemma, we will show how it implies Theorem \ref{thm:genclass}. The proof is similar to the proof of the $m=2$ case in \cite{dope}, with Lemma \ref{lem:lekoneLinIndep} replacing the result of Gessel-Viennot.

\begin{proof}[Proof of Theorem \ref{thm:genclass}]
First, we show that being \lek\ is necessary. Let $\Lambda$ be a generic $m$-tuple. Suppose for some $\mathbf{S}$ such that $M(\mathbf{S})$ is not \lek, there exists a polynomial $P$ such that $D_P(\Lambda)=M(\mathbf{S})$.

\

Take the smallest $k$ such that the last $k+1$ columns contain at least $k+1$ nonzero entries. The linear forms corresponding to these entries are linear functions of $a_n,a_{n-1},\ldots,a_{n-k}$ and are linearly independent by Lemma \ref{lem:lekoneLinIndep}. Hence, $a_n=0$, contradicting the assumption that $P$ is degree-$n$.

\

To show that the condition is sufficient, given any \lek\ matrix $M$, we prepend columns where the top two entries are $1$ and the remaining entries are $0$ such that the resulting matrix $M'$ has exactly $n+c$ nonzero entries, where $c$ is the number of prepended columns. We have that $M'$ is \lek.

\

Let $\mathbf{S}$ be such that $M(\mathbf{S})=M'$. By Lemma \ref{lem:lekoneLinIndep}, $P^{\mathbf{S}}(\Lambda)$ consists of $n+c$ independent linear forms. Appending the linear form $P^{(n+c)}(\lambda_1)$ corresponds to adding a one to the last column of $M'$, which makes $M'$ remain \lekone, so by Lemma \ref{lem:lekoneLinIndep}, $P^{\mathbf{S}}(\Lambda) \cup \{P^{(n+c)}(\lambda_1)\}$ is a set of $n+c+1$ independent linear forms. Hence, there is some nonzero polynomial $P_0$ with degree at most $n+c$ such that $P_0^{(n+c)}(\lambda_1)=1$, and for all $i \in [m]$ and $s  \in S_i$, we have $P_0^{(s)}(\lambda_i)=0$. Since $P_0^{(n+c)}(\lambda_1)=1$, we have that $P_0$ has degree exactly $n+c$.

\

We next show that we cannot have $P_0^{(s)}(\lambda_i)=0$ for any $i \in [m]$ and $s \not\in S_i$. Consider $\mathbf{S}' = (S_1,\ldots,S_{i-1},S_i \cup\{ s\}, S_{i+1},\ldots,S_m)$. If $P_0^{(s)}(\lambda_i)=0$, then $P_0$ is zero on every element of $P^{\mathbf{S}'}(\Lambda)$. However, $M(\mathbf{S})$ is \lekone, so $P^{\mathbf{S}'}(\Lambda)$ contains $n+c+1$ linearly independent linear functions by Lemma \ref{lem:lekoneLinIndep}. This forces $P_0$ to be the zero polynomial, which is a contradiction. Thus, $D_{P_0}(\Lambda)$ is exactly $M'$, and hence $D_{P_0^{(c)}}(\Lambda)$ is the desired matrix.
\end{proof}

\subsection{Demonstration of Proof Technique}

As the proof of Lemma \ref{lem:lekoneLinIndep} is fairly technical, we provide a demonstration of the proof for a small \lekone\ matrix.

\

Consider the $\mathbf{S}$ corresponding to the following $3 \times 6$ matrix, which is \lekone:
\[M(\mathbf{S})=\begin{bmatrix}1 & 1 & 0 & 0 & 1 & 0  \\ 0 & 0 & 0 & 1 & 0 & 0 \\ 1 & 0 & 0 & 0 & 1 & 0 \end{bmatrix}.\]
We will show that if $\Lambda = (0,1,t)$ for some transcendental $t$, then $P^{\mathbf{S}}(\Lambda)$ is linearly independent.

\

We want to show that for generic $t$, the linear forms
	\[P(0),P'(0),P^{(4)}(0),P^{(3)}(1),P(t),P^{(4)}(t)\]
	are linearly independent. It suffices to show that the result holds for at least one transcendental $t$, as we can then take a $\QQ$-automorphism mapping $t$ to any transcendental number of our choosing. The key idea is to check the special case of $t$ very close to $0$.
	
	\
	
	Taking linear combinations, we find that the span of the aforementioned linear forms is the same as the span of
	\[P(0),P'(0),P^{(4)}(0),P^{(3)}(1),\frac{P(t)-P(0)-tP'(0)}{t^2/2},\frac{P^{(4)}(t)-P^{(4)}(0)}{t}.\]
	The last two forms are polynomials in $t$, and can hence be continuously extended to $t=0$. At $t=0$, the forms are equal to
	\[P(0),P'(0),P^{(4)}(0),P^{(3)}(1),P^{(2)}(0),P^{(5)}(0).\]
	These linear forms correspond to the $2 \times 6$ \lekone\ matrix
	\[\begin{bmatrix}1 & 1 & 1 & 0 & 1 & 1 \\ 0 & 0 & 0 & 1 & 0 & 0 \end{bmatrix},\]
	and hence are linearly independent by the $m=2$ case of the theorem. By the fact that linear independence is equivalent to nonzero determinant, we can extend the linear independence of
	\[P(0),P'(0),P^{(4)}(0),P^{(3)}(1),\frac{P(t)-P(0)-tP'(0)}{t^2/2},\frac{P^{(4)}(t)-P^{(4)}(0)}{t}\]
	from $t=0$ to all $t$ in some neighborhood of $0$. Hence, they are linearly independent for some transcendental $t$, implying the desired result.
	
\

One can view the argument above as showing that we can combine the roots at $0$ and $t$, and hence combine the corresponding rows of the matrix.

\

The proof in the general case consists of two parts. First, we prove a claim generalizing the choice of linear combinations in the above proof. We then generalize the $t \to 0$ argument, allowing us to combine the $\lambda_i$'s under certain conditions. Once this is done, Lemma \ref{lem:lekoneLinIndep} follows from repeatedly combining the $\lambda_i$'s.

\subsection{Derivatives as Linear Combinations}

We first, using Theorem \ref{thm:binomlinindep}, find linear combinations that limit to derivatives.

\begin{lemma}\label{lem:derivativelincomb}
	Fix $d \in \ZZ_{\geq 0}$. Let $\mathbf{S}=(S_1,S_2)$, where $S_1,S_2 \subset [0,d]$. Suppose that in $M(\mathbf{S})$, for all $0 \leq k \leq d$, there are at most $k+1$ nonzero entries in columns $[d-k,d]$, with equality holding for $k=d$. Then there exist constants $c_{s,1}$ for $s \in S_1$ and $c_{s,2}$ for $s \in S_2$ such that the following holds:
	
	\

	For every polynomial $P$, there exists a polynomial $Q \in \CC[\lambda,\eps]$ such that
	
	\begin{equation}\label{eqn:linspan}
	P^{(d)}(\lambda)=\sum_{s \in S_1} c_{s,1} \frac{P^{(s)}(\lambda)}{s!} \eps^{s-d} +\sum_{s \in S_2} c_{s,2}  \frac{P^{(s)}(\lambda+\eps)}{s!}\eps^{s-d}+\eps Q(\lambda,\eps)\tag{$\star$}
	\end{equation}
	holds for all $\lambda,\eps \in \CC$.
\end{lemma}
\begin{proof}
Let $P(x)=a_0(x-\lambda)^0+\cdots+a_d(x-\lambda)^d+(x-\lambda)^{d+1}Q_0(x-\lambda)$. We will view the $a_i$ as variables. We can write $P^{(s)}$ evaluated at $\lambda$ and $\lambda+\eps$ in the basis of the $a_i$ as follows:
\[\frac{P^{(s)}(\lambda)}{s!} \eps^{s-d}= \eps^{s-d}a_{s} \quad\text{and}\quad \frac{P^{(s)}(\lambda+\eps)}{s!}=\sum_{t=s}^n \binom{t}{s}\eps^{t-d}a_t.\]
Hence, the right-hand side of \ref{eqn:linspan} contains only linear combinations of $\set{\eps^{t-d}a_t \mid t \in [0,n]}$. For any choice of the $c_{s,2}$'s, we can pick $Q$ to make the coefficient of $\eps^{t-d}a_t$ zero for $t>d$, and we can pick $c_{t,1}$ to make the coefficient of $\eps^{t-d}a_t$ zero for $t \in S_1$. Hence, it suffices to pick $c_{s,2}$ such that, for all $t \in [0,d]\setminus S_1$, the coefficient of $\eps^{t-d}a_t$ in the right-hand side of \ref{eqn:linspan} matches the corresponding coefficient in the left-hand side.

\ 

Let $\mathcal{G} = [0,d] \setminus S_1$ and $\mathcal{H} = S_2$. For all $g \in \mathcal{G}$, the coefficient of $\eps^{g-d}a_g$ in some term in the right-hand side of (\ref{eqn:linspan}) is
\[\sum_{s \in \mathcal{H}}\binom{g}{s}c_{s,2}\]

The column condition implies $\abs{S_1 \cap [0,c]}+\abs{S_2 \cap [0,c]} \geq c+1$ for all $c$, so $\abs{\mathcal{G} \cap [0,c]} \leq \abs{\mathcal{H} \cap [0,c]}$ for all $c$. Now, by Theorem \ref{thm:binomlinindep}, we can find constants $c_{h,2}$ such that for all $g$ in $\mathcal{G}$,

\[\sum_{h \in \mathcal{H}}\binom{g}{h}c_{h,2}=\begin{cases}0\quad \text{if }g \neq d \\ d!\quad \text{if }g=d,\end{cases}\]

as desired.
\end{proof}

\subsection{Combining Roots}
We make use of the following fact, which allows us to take limits of linear dependences:

\begin{proposition}\label{prop:opencondition}
	Let $\gamma_1,\ldots,\gamma_\ell$ be continuous maps from $\CC$ to some complex vector space. If $\gamma_1(0),\ldots,\gamma_\ell(0)$ are linearly independent, then $\gamma_1(t),\ldots,\gamma_\ell(t)$ are linearly independent for all $t$ in some neighborhood of $0$.
\end{proposition}

This follows from the fact that, in a given complex vector space, linearly independent $\ell$-tuples of vectors form an open set.

\

We now prove the claim allowing us to ``combine" $\lambda_1$ and $\lambda_m$. This will be the inductive step in our proof of Lemma \ref{lem:lekoneLinIndep}.
\begin{lemma}\label{lem:safeLinIndep}
	Suppose $m>1$, and suppose that $M(\mathbf{S})$ is \lekone. Define $\Lambda ' = \Lambda \setminus \{\lambda_m\}$ and $\mathbf{S}'=(S'_1,\ldots,S'_{m-1})$, where
	\begin{itemize}
		\item $S'_i = S_i$ for $2 \leq i \leq m-1$
		\item $s \in S'_1$ if and only if for some $t\leq s$, we have $\abs{S_1 \cap [t,s]} + \abs{S_m \cap [t,s]} \geq s-t+1$.
	\end{itemize}
	Then, we have:
	\begin{enumerate}
		\item[(a)] the matrix $M(\mathbf{S}')$ is \lekone,
		\item[(b)] if $P^{\mathbf{S}'}(\Lambda')$ is linearly independent, then so is $P^{\mathbf{S}}(\Lambda)$, and
		\item[(c)] if $\Lambda$ is generic, so is $\Lambda'$.
	\end{enumerate}
\end{lemma}

\begin{proof}
The proof of (c) is clear. We will begin by proving (a).

\

Consider the row vector obtained as follows:

\begin{enumerate}
	\item Add the first and $m$th rows of $M(\mathbf{S})$, obtaining some vector $v$.
	\item We index the components of $v$ with $0,1,\ldots,n$. If, for some $j$, the $j$th component of $v$ is at least $2$, subtract $1$ from the $j$th component and add $1$ to the $(j+1)$th component.
	\item Repeat the previous step until it cannot be repeated anymore.
\end{enumerate}

This aligns with the intuition of combining the roots of row $1$ and row $m$, since we expect a double root of $P^{(j)}$ to become a root of $P^{(j)}$ and a root of $\paren{P^{(j)}}'=P^{(j+1)}$. We claim that, no matter which choices are made during step 2, this process will always result in the row vector corresponding to $S'_1$.

\

Let $v_i$ denote the vector obtained in step 1, and let $v_f$ denote the vector obtained at the end of the process. We use $v$ to denote an arbitrary vector at any point in the process.

\

For each $s$, consider the quantity $C_s(v) \vcentcolon= \max_{t \leq s}\paren{v \cdot \mathbf{1}_{[t,s]}-(s-t+1)}$, where $\mathbf{1}_{[t,s]}$ is the row vector that is $1$ on the columns $[t,s]$ and $0$ elsewhere. During any application of step 2 above, $C_s(v)$ is unchanged if $j<s$, is decreased by $1$ if $j=s$, and is unchanged if $j>s$. In the $j=s$ case, taking $t=s$, we must have that $C_s(v) \geq 0$ after the application of step 2. Hence, $C_s(v_i) \geq 0$ if and only if $C_s(v_f) \geq 0$.

\

Since $M(\mathbf{S})$ is \lekone, we have $C_n(v_i)\leq 0$. Hence, we must also have $C_n(v_f) \leq 0$ -- in particular, the $n$th component of $v_f$ cannot be more than $1$. By definition, the $j$th component of $v_f$ cannot be more than $1$ for any $j<n$. Thus, all components of $v_f$ are either $0$ or $1$.

\

We have $v_i \cdot \mathbf{1}_{[t,s]}=\abs{S_1 \cap [t,s]} + \abs{S_m \cap [t,s]}$, so $C_s(v_i)\geq 0 $ if and only if $s \in S'_1$. Since $C_s(v_f)\geq 0$ is equivalent to both $C_s(v_i) \geq 0$ and the $j$th component of $v_f$ being nonzero, $v_f$ is the row vector corresponding to $S'_1$.

\

As a corollary of this alternate characterization, we have $\abs{S_1}+\abs{S_m} = \abs{S'_1}$, and hence $P^{\mathbf{S}}(\Lambda)$ and $P^{\mathbf{S}'}(\Lambda')$ have the same number of elements. We also have that if $M(\mathbf{S})$ is \lekone, then so is $M(\mathbf{S}')$, as all steps of the above process preserve the property that the sum of the entries of the last $k+1$ columns is at most $k+1$ for all $k$.

\

We now prove (b). Recall that we use $P$ to denote a general polynomial of degree at most $n$. We claim that for each $s' \in S'_1$, there exist a polynomial $Q_{s'}$, independent of $\eps$ but possibly dependent on $\Lambda$ and $P$, and constants $c_{s,i}$, possibly dependent on $\eps$, such that
\[P^{(s')}(\lambda_1) -\eps Q_{s'}(\eps)=  \paren{\sum_{s \in S_1} c_{s,1} P^{(s)}(\lambda_1)}+\paren{\sum_{s  \in S_m} c_{s,m} P^{(s)}(\lambda_1+\eps)}.\]

If $s \in S_1$, the result is clear. Otherwise, take the largest $t$ such that $\abs{S_1 \cap [t,s']}+\abs{S_m \cap [t,s']} \geq s'-t+1$. For this $t$, we have $\abs{S_1 \cap [t,s']}+\abs{S_m \cap [t,s']} = s'-t+1$, as otherwise $t+1$ would have the same property. Applying Lemma \ref{lem:derivativelincomb} to the polynomial $P^{(t)}$, the degree $d=s'-t$, $\lambda=\lambda_1$, $\eps=\lambda_m-\lambda_1$, and the sets $\mathbf{S}=(\{s-t \mid s \in S_1, s\geq t\},\{s-t \mid s \in S_m, s\geq t\})$ gives the desired $Q$ and $c_{s,i}$.

\

By our assumptions, $P^{\mathbf{S'}}(\Lambda')$, which is \[\paren{P^{(S_2',\ldots,S_{m-1}')}((\lambda_2,\ldots,\lambda_{m-1}))} \bigcup \paren{\bigcup_{s' \in S'_1}\set{P^{(s')}(\lambda_1)-\eps Q_{s'}(\eps)}}\]
evaluated at $\eps=0$, is linearly independent. Every linear form in the above set is continuous in $\eps$, and hence the above set is linearly independent for some transcendental $\eps \neq 0$ by Proposition \ref{prop:opencondition}. Its span for this $\eps$ is a subspace of the span of $P^{\mathbf{S}}(\Lambda)$, and has dimension $\abs{P^{\mathbf{S}'}(\Lambda')}=\abs{P^{\mathbf{S}}(\Lambda)}$. Hence, $P^{\mathbf{S}}(\Lambda)$ is linearly independent.
\end{proof}

We can now repeatedly combine elements to prove that all \lekone\ matrices are linearly independent.

\begin{proof}[Proof of Lemma \ref{lem:lekoneLinIndep}]
	We proceed by induction on $m$. The base case of $m=1$ is trivial. The induction step follows immediately from Lemma \ref{lem:safeLinIndep}.
\end{proof}

\subsection{Enumeration}

The enumeration of generic dope matrices follows from a direct application of the cycle lemma.

\begin{definition} \normalfont We say that a sequence of ones and zeroes is \emph{$t$-dominating} if for every $\ell>0$, the number of zeroes among the first $\ell$ entries is more than $t$ times the number of ones.
\end{definition}

The cycle lemma allows us to count the number of $t$-dominating sequences with a given number of ones and zeroes.

\begin{theorem}\cite[Cycle Lemma]{cycle} \label{thm:cycle}
	 Let $a,b,t$ be nonnegative integers with $a \geq tb$. For any sequence $p_1,\ldots,p_{a+b}$ of $a$ zeroes and $b$ ones, exactly $a-tb$ of the cyclic shifts of the sequence are $t$-dominating.
\end{theorem}

\begin{proof}[Proof of Theorem \ref{thm:enumeration}]
	To prove the assertion for fixed $k$, consider the map from $\paren{c_{ij}}_{i \in [m],j \in [0,n]} \in \D_n^{\gen(m)}$ to $(m-1)$-dominating sequences $a_1,\ldots,a_{m(n+1)}$ given by $c_{ij} \leftrightarrow a_{m (n+1-j)+i}$. By Theorem \ref{thm:genclass}, this is a bijection between elements of $\D_n^{\gen(m)}$ with $k$ ones and $m$-dominating sequences with $k$ ones.
	
	\
	
	By Theorem \ref{thm:cycle}, of the $\binom{(n+1)m}{k}$ length-$m(n+1)$ sequences with $k$ ones and $m(n+1)-k$ zeroes,
	\[\frac{n+1-k }{n+1} \dbinom{(n+1)m}{k}=\binom{(n+1)m-1}{k}-(m-1)\binom{(n+1)m-1}{k-1}\]
	of them are $(m-1)$-dominating, using the convention $\binom{N}{-1}=0$, implying the assertion for fixed $k$. Summing over $0 \leq k \leq N$ gives the desired formula for $\abs{\D_n^{\gen(m)}}$.
\end{proof}

\begin{remark}
	When $m=2$, the size of $\D_n^{\gen(m)}$ simplifies to $\binom{2n+1}{n}$. In this case, a direct counting argument is possible. The earlier map gives a bijection between $\D_n^{\gen(2)}$ and $1$-dominating $\{0,1\}$ sequences. Deleting the first element of the sequence and treating zeroes as ups and ones as downs, the $1$-dominating $\{0,1\}$ sequences are in bijection with length-$(2n+1)$ left factors of Dyck paths, of which there are $\binom{2n+1}{n}$.
\end{remark}

From the formula in Theorem \ref{thm:enumeration}, we can find good closed-form estimates for $\abs{D_n^{\gen(m)}}$. For $m=1,2$, we have the exact formulas $2^n$ and $\binom{2n+1}{n}$, respectively. For larger $m$, we have the following:

\begin{corollary}\label{cor:genericbounds}
	For $m \geq 3$ and $n \geq 1$, we have
	\[\frac{1}{n+1}\dbinom{(n+1)m}{n} \leq \abs{\D_n^{\gen(m)}} \leq \paren{1+\frac{1}{m-2}}^2 \frac{1}{n+1}\dbinom{(n+1)m}{n}.\]
\end{corollary}

\begin{proof}
	We use the formula in \ref{thm:enumeration}. The lower bound follows by considering the $k=n$ term. For the upper bound, note that the $k=n-\ell$ term is
	
\begin{align*}
\frac{\ell+1}{n+1}\binom{(n+1)m}{n-\ell} &\leq \frac{1}{n+1} \binom{(n+1)m}{n} \left[ \frac{(\ell+1)n^\ell}{((n+1)m-n)^\ell} \right]\\
&\leq \frac{1}{n+1} \binom{(n+1)m}{n} \left[ (\ell+1)(m-1)^{-\ell}\right].
\end{align*}
Summing over $0 \leq \ell \leq n$, we have
\[\abs{\D_n^{\gen(m)}} \leq \frac{1}{n+1} \binom{(n+1)m}{n} \cdot  \paren{\sum_{\ell=0}^\infty (\ell+1)(m-1)^{-\ell}}=\paren{1+\frac{1}{m-2}}^2 \frac{1}{n+1}\dbinom{(n+1)m}{n},\]
as desired.
\end{proof}

\section{$\D_n^m$ for Small $m$}

In this section, we will prove Theorem \ref{thm:normalmbounds}. We also remark that, combining our lower bound and upper bound, we have the following asymptotic estimate for $\abs{\D_n^m}$ when $m=o(n)$:

\begin{theorem}[Corollary of Theorem \ref{thm:normalmbounds}]\label{thm:correct}
	For $m(t),n(t)\fromto{\ZZ_{>0}}{\ZZ_{>0}}$ satisfying $n(t) \to \infty$ and $m=o(n)$, we have \[ \log \abs{\D_n^m} = (1+o(1))\log \binom{mn}{n}.\]
\end{theorem}

\subsection{Upper Bound}

We will prove the following upper bound.

\begin{proposition}\label{prop:upperBound}
	For $m(t),n(t)\fromto{\ZZ_{>0}}{\ZZ_{>0}}$ satisfying $n(t) \to \infty$, we have \[\log \abs{\D_n^{m}} \leq (1+o(1))\log \paren{n^{2m}\binom{mn}{n}}.\]
\end{proposition}
Let $f_1,\ldots,f_T$ be a sequence of polynomials in $N$ variables.
\begin{definition}
	\normalfont We define a \emph{zero-pattern} to be a subset $S$ of $[T]$ of the form
	\[S = \left\{ a | f_a(u)=0\right\}\]
	for some $u \in \CC^N$.
\end{definition}
Our proof closely follows the proof of the following result of R\'{o}nyai, Babai, and Ganapathy, which provides a bound on the number of zero-patterns of general sequences of polynomials:
\begin{theorem}\cite[Theorem 4.1]{zeroPatterns}\label{thm:zeroPatterns}
	Let $f_1,\ldots,f_T$ be a sequence of polynomials in $N$ variables, where each polynomial has degree at most $d$. For any $t$, the number of zero-patterns is at most
	\[\paren{\binom{T}{0}+\cdots+\binom{T}{t}}+\binom{N+(T-t-1)d}{N}.\]
\end{theorem}

Our proof uses the key ideas from the proof of Theorem \ref{thm:zeroPatterns}, with the main difference being that we can estimate the degrees of polynomials more carefully in the specific case of the sequence $\set{P^{(j)}(\lambda_i)}$.

\begin{proof}[Proof of Proposition \ref{prop:upperBound}]
	Since we assume that $P$ is of degree $n$, we may assume $P$ is monic, as scaling does not affect roots. Let $P=x^n+a_{n-1}x^{n-1}+\cdots+a_0$ and $\Lambda=(\lambda_1,\ldots,\lambda_m)$, where $a_j,\lambda_i \in \CC$ are variables. We view each $P^{(j)}(\lambda_i)$ as a polynomial in $(a_0,\ldots,a_{n-1},\lambda_1,\ldots,\lambda_m)$. Let $S_1,\ldots,S_M \subset  [m] \times [0,n-1]$ denote the zero-patterns of $\{P^{(j)}(\lambda_i)\}$, where we may exclude $j=n$ since $P^{(n)}$ is the constant polynomial $(n-1)!$.
	
	\
	
	Call a zero-pattern \emph{large} if it has size larger than $n$, and \emph{small} otherwise. The number of small zero-patterns is at most
	\[\binom{mn}{0}+\binom{mn}{1}+\cdots+\binom{mn}{n} \leq n \binom{mn}{n}.\]
	For each large zero-pattern $S_k$, consider the polynomial
	\[Q_k(P,\Lambda) \vcentcolon= \prod_{(i,j) \not\in  S_k} P^{(j)}(\lambda_i).\]
	We claim that the $Q_k$ are linearly independent (this is proven in \cite[Theorem 1.1]{zeroPatterns}). Suppose, for the sake of contradiction, that some linear combination $\sum_k \alpha_k Q_k$ is identically zero, where the $\alpha_k$ are not all zero. Consider some index $\ell$ that maximizes $\abs{S_\ell}$ over all $\ell$ with $\alpha_\ell\neq0$. Evaluating $\sum_k \alpha_k Q_k$ at the $(P,\Lambda)$ corresponding to the zero pattern $S_\ell$ gives $\alpha_\ell=0$, which is a contradiction, as desired.
	
	\
	
	Hence, the $Q_k$'s are linearly independent. Furthermore, all of the monomials of the $Q_k$'s corresponding to large zero-patterns are in the set
	\[\set{ a_{0}^{b_{0}}\cdots a_{n-1}^{b_{n-1}} \cdot \lambda_1^{c_1}\cdots \lambda_m^{c_m}\mid b_0+b_1+\cdots+b_{n-1} \leq mn-n,\ 0 \leq c_i \leq \frac{n^2+n}{2}  }.\]
	This is a set of size $\binom{mn}{n} \cdot \paren{\frac{n^2+n+2}{2}}^m$, so the $Q_k$'s corresponding to large zero-patterns lie in a space of dimension $\binom{mn}{n} \cdot \paren{\frac{n^2+n+2}{2}}^m$. Hence, we have
	\[\abs{\D_n^m} \leq \binom{mn}{n} \cdot \paren{\paren{\frac{n^2+n+2}{2}}^m+n},\]
	giving the desired bound.
\end{proof}

\subsection{Lower Bound}

We now establish the lower bound on $\abs{\D_n^m}$ from Theorem \ref{thm:normalmbounds}.

\begin{proposition}\label{prop:mixedConstruct}
	For $m(t),n(t)\fromto{\ZZ_{>0}}{\ZZ_{>0}}$ satisfying $n(t) \to \infty$ and $1<m(t) \leq \frac{n(t)^2+n(t)}{2}$, we have \[\log \abs{\D_n^m} \geq (1+o(1))\log \paren{n^{m}\binom{mn}{n}}.\]
\end{proposition}

The idea behind the construction is, for some well chosen $a$, to start with an $a \times (n+1)$ matrix $M \in \D_n^{\gen(a)}$, to pick $P$ and $\Lambda$ such that $D_P(\Lambda)=M$, and then to append $m-a$ elements to $\Lambda$. We first prove a claim allowing us to find $P$ such that many distinct rows can be appended to $D_P(\Lambda)$.

\

Call an $m \times (n+1)$ matrix \emph{$T$-limited} if each row has at most $T$ ones. Call an $m \times (n+1)$ \lek\ matrix \emph{saturated} if it has exactly $n$ ones in total. We let $C(m,n,T)$ denote the number of $m \times (n+1)$ \lek\ $T$-limited saturated matrices.

\begin{proposition}\label{prop:limited}
	Let $\Lambda=(\lambda_1,\ldots,\lambda_a)$ be affinely algebraically independent, and $M$ be an $a \times (n+1)$ \lek\ $T$-limited saturated matrix. Then there is a degree-$n$ polynomial $P$ such that $D_P(\Lambda)=M$. Furthermore, this polynomial $P$ has the property that for any $\lambda \in \CC$, at most $T$ of the entries of $D_P((\lambda))$ are one.
\end{proposition}

\begin{proof}
	Using Theorem \ref{thm:genclass}, we can pick some polynomial $P$ such that $D_P(\Lambda)=M$. Fix an arbitrary $\lambda \in \CC$. We claim that for some $b \in [a]$, the $a$-tuple $\Lambda_b$ obtained from $\Lambda$ by replacing $\lambda_b$ with $\lambda$ is affinely algebraically independent. If $\Lambda$ with $\lambda$ appended is already affinely algebraically independent, the result is clear.
	
	\
	
	Otherwise, let $Q_0$ be a minimum-degree affine algebraic dependence of $(\lambda,\lambda_1,\ldots,\lambda_{a})$. We claim that $Q_0$ divides all algebraic dependences. Suppose, for the sake of contradiction, that $Q$ is another affine algebraic dependence such that $Q_0$ does not divide $Q$. Viewing $Q_0,Q$ as polynomials in the first variable and taking the resultant gives a nonzero polynomial $R$ with $R(t_1+t_1\lambda_1,\ldots,t_1+t_2\lambda_{a})=0$ for all $t_1,t_2$, contradicting the affine algebraic independence of $\Lambda$.
	
	\
	
	Now, if we choose $b$ such that $x_b$ appears in $Q_0$, we find that $\Lambda_b$ is affinely algebraically independent. For this $b$, by Theorem \ref{thm:genclass}, $D_P(\Lambda_b)$ must have at most $n$ ones, so the number of ones in $D_P((\lambda))$ is at most the number of ones in the row of $\lambda_b$, which is at most $T$ by assumption, as desired.
\end{proof}

Now, we execute the construction mentioned earlier.

\begin{lemma}\label{lem:grossbound}
	For integers $m,n,a,T$ with $0 \leq a \leq m \leq \frac{n^2+n}{T^2+T}$, we have
	\[\abs{\D_n^m} \geq C(a,n,T) \cdot \paren{\frac{n+1}{e(T^2+T)}-\frac{a}{en}}^{m-a}.\]
\end{lemma}

\begin{proof}
	For each $a \times (n+1)$ \lek\ $T$-limited saturated matrix $M$, we construct many matrices in $\D_n^m$ whose top $a$ rows are $M$ as follows:
	\begin{enumerate}
		\item Pick an algebraically independent $a$-tuple of complex numbers $\lambda_1,\ldots,\lambda_a$. Pick $P$ as in Proposition \ref{prop:limited} so that $D_P((\lambda_1,\ldots,\lambda_a))=M$.
		\item For each $a+1 \leq i \leq m$, iteratively choose $\lambda_i$ to be any value not equal to any of $\lambda_1,\ldots,\lambda_{i-1}$ so that $D_P((\lambda_i))$ is nonzero.
	\end{enumerate}
	We will show that there must be many possible choices for each $\lambda_i$ in Step 2. We have that $\lambda$ is a root of the degree-$\paren{n(n+1)/2}$ polynomial $\prod_{j=0}^n P^{(j)}$ of multiplicity $\sum_t \frac{t(t+1)}{2}$, where the summation is over all lengths of runs of ones, with multiplicity, in $D_P((\lambda))$. Since $D_P((\lambda))$ has at most $T$ nonzero entries, and $\frac{x(x+1)}{2}$ is convex, $\lambda$ is a root of the degree-$\paren{n(n+1)/2}$ polynomial $\prod_{j=0}^n P^{(j)}$ of multiplicity at most $T(T+1)/2$. Note that $D_P((\lambda))$ is nonzero if and only if $\lambda$ is a root of $\prod_{j=0}^n P^{(j)}$. Hence, the number of possible choices for $\lambda_{i+1}$ is at least
	\[\ceil{\frac{n(n+1)/2}{T(T+1)/2}}-i = \ceil{\frac{n^2+n}{T^2+T}}-i.\]
	Since $P^{(j)}$ has at most $n$ roots for any $j$, there are at most $n$ possible values of $\lambda_{i+1}$ corresponding to the same not-all-zero row. Thus, given a choice of $\lambda_1,\ldots,\lambda_i,$ there are at least $\frac{1}{n}\paren{ \ceil{\frac{n^2+n}{T^2+T}}-i}$ possibilities for the $(i+1)$th row. For each $a \times (n+1)$ \lek\ $T$-limited saturated matrix, this construction gives at least
	\[\prod_{i=a}^{m-1} n^{-1}\paren{\ceil{\frac{n^2+n}{T^2+T}}-i}\]
	elements in $\D_n^m$. For any nonnegative integers $b<k$, we have
	\[\paren{k(k-1)\cdots(k+1-b)}^{1/b} \geq \paren{k(k-1)\cdots(1)}^{1/k} \geq \frac{k}{e},\]
	where the first inequality follows from the fact that the left-hand side is decreasing in $b$, and the second inequality is justified in Section \ref{sec:prelims}. Taking $k=\ceil{\frac{n^2+n}{T^2+T}}-a$ and $b=m-a$, we find that the number of elements in $\D_n^m$ is at least
	\[C(a,n,T)\prod_{i=a}^{m-1} n^{-1}\paren{\ceil{\frac{n^2+n}{T^2+T}}-i} \geq C(a,n,T) n^{a-m} \paren{\frac{\ceil{\frac{n^2+n}{T^2+T}}-a }{e} }^{m-a},\]
	which exceeds the bound required.
\end{proof}

To prove Proposition \ref{prop:mixedConstruct}, it remains to analyze the size of the bound in Lemma \ref{lem:grossbound}.

\begin{proposition}\label{prop:limitlowerbound}
	Suppose positive integers $a$ and $T$ with $T$ divisible by $3$ satisfy $T \leq n$ and $(a-2)T/3 > n$. Then we have
	\[C(a,n,T) \geq \frac{1}{n+1}\binom{an}{n}\paren{1-a(n+1)2^{-T/3}}.\]
\end{proposition}

\begin{proof}
	The number of $a \times (n+1)$ \lek\ matrices with exactly $t$ ones in the first row and $n$ ones in total is at most
	\[f(t) \vcentcolon= \binom{n}{t} \binom{(a-1)n}{n-t}.\]
	 Here, the first term counts the number of ways to distribute $t$ ones into the first row and the second term counts the number of ways to distribute the remaining ones. Here, we are ignoring the condition given by $M$ being \lek, but use the fact that the last column cannot have ones.
	 
	 \
	
	We have, for all $t \leq n-1$,
	\[f(t+1) = f(t) \cdot \frac{n-t}{t+1} \cdot \frac{n-t}{(a-2)n+t} < f(t)\frac{n}{t(a-2)}\leq f(t) \frac{T/3}{t}.\]
	For $t \geq 2T/3$, the function $f$ decays by a factor of at least $2$ when $t$ increases by $1$. The number of $a \times (n+1)$ \lek\ $T$-limited saturated matrices is at least, by Corollary \ref{cor:genericbounds},
	\[\frac{1}{n+1}\binom{a(n+1)}{n}-a \paren{\sum_{t \geq T+1}f(t)},\]
	as the number of $a \times (n+1)$ $\{0,1\}$ matrices with at least $t$ ones in a given row is at most $f(t)$. Since $f(T+i) \leq 2^{-i-T/3}f(2T/3)$, we get
	\[C(a,n,T) \geq \frac{1}{n+1}\binom{a(n+1)}{n}-a2^{-T/3} f(2T/3).\]
	
	By Vandermonde's identity, we have
	
	\[\binom{a(n+1)}{n} \geq \binom{an}{n} =\sum_{i=0}^n f(i)\geq f(2T/3).\]
	Hence,
	\[C(a,n,T) \geq \frac{1}{n+1}\binom{a(n+1)}{n}-a2^{-T/3} f(2T/3)\geq \frac{1}{n+1}\binom{a(n+1)}{n}\paren{1-a(n+1)2^{-T/3}},\]
	which exceeds the desired bound.
\end{proof}

We are now ready to prove Proposition \ref{prop:mixedConstruct}. 

\begin{proof}[Proof of Proposition \ref{prop:mixedConstruct}]
For $\frac{n}{(\log n)^2} \geq m$, use the lower bound $\abs{\D_n^m} \geq \abs{\D_n^{\gen(m)}}$ to get
\[\log \abs{\D_n^m} \geq \log \abs{\D_n^{\gen(m)}} \geq \log \paren{\frac{1}{n+1}\binom{m(n+1)}{n}} =(1+o(1))\log \paren{n^m\binom{mn}{n}},\]
where the second inequality follows from Corollary \ref{cor:genericbounds} and the last estimate of the error term follows from bounds $\log (n^m) \leq n/\log n=o(n)$ and $\log \binom{mn}{n} \geq \log \binom{2n}{n}=\Theta(n)$ .

\

We use Lemma \ref{lem:grossbound} for the remaining regions. For $m \geq  n\sqrt{\log n}$, take $a=n$ and $T=1$, so $C(a,n,T) \geq 1$. We have assumed $m \leq \frac{n^2+n}{2} =  \frac{n^2+n}{T^2+T}$, so we may apply Lemma \ref{lem:grossbound}, which gives
\[\log \abs{\D_n^m} \geq (m-n) \log \paren{\frac{n+1}{2e}-\frac{1}{2}} = (1+o(1))m \log n= (1+o(1))\log (n^m).\]
This is equal to $(1+o(1))\log \paren{n^{m}\binom{mn}{n}}$ because we can estimate $\log \binom{mn}{n}$ by
 \[\log \binom{mn}{n} \leq \log \paren{(mn)^n} \leq 3n \log n = o(m\sqrt{\log n}).\]

Finally, for $\frac{n}{(\log n)^2} \leq m \leq n \sqrt{\log n}$, take $T/3 = 10 \ceil{\log n}^3$ and $a = \ceil{\frac{m}{T/3}}+3$. For large enough $n$, we have $m \leq \frac{n^2+n}{T^2+T}$, so we can apply Lemma \ref{lem:grossbound}. For large enough $n$, we have:
\[\frac{n+1}{e(T^2+T)}-\frac{a}{en} \geq \frac{n}{100T^2}\quad\text{and}\quad  m-a \geq m\paren{1-\frac{1}{\log n}},\]
allowing us to estimate that, using Proposition \ref{prop:limitlowerbound}, $\log \paren{\frac{\abs{\D_n^m}}{C(a,n,T)}} \geq (1+o(1)) \log \paren{n^m}$. For sufficiently large $n$, we also have
\[\binom{an}{n} \geq (\log n)^{-100 n}\binom{mn}{n} \quad\text{and}\quad 1-a(n+1)2^{-T/3}\geq \frac{1}{2},\]
which will allow us to estimate $\log C(a,n,t) \geq (1+o(1))\log \paren{\binom{mn}{n}}$. For $n$ large enough for the aformentioned inequalities to hold, putting everything together gives
\[\abs{\D_n^m}\geq \frac{1}{2(n+1)}(\log n)^{-100 n}\binom{mn}{n} \cdot \paren{\frac{n}{100 T^2}}^{m\paren{1-\frac{1}{\log n}}},\]
which is $\paren{\binom{mn}{n} n^m}^{1+o(1)}$, as desired.
\end{proof}

\section{$\D_n^m$ for Large $m$}

The goal of this section is to prove Theorem \ref{thm:largeasymptotics}.

\

Let $V(m,n)$ denote the number of elements in $\D_n^m$ with no all-zero rows. As $P^{(j)}$ has at most $n-j$ roots for each $j$, the $j$th column of any element of $\D_n^m$ has at most $n-j$ ones, and so $V(m,n)=0$ for all $m > \frac{n^2+n}{2}$.

\begin{proposition}\label{prop:polynom}
	For all $m,n$, we have
	\[\abs{\D_n^m} = \sum_{k=0}^{\frac{n^2+n}{2}} \binom{m}{k}V(k,n).\]
\end{proposition}
\begin{proof}
	The right-hand side counts $\abs{\D_n^m}$ by conditioning on the number $k$ of not-all-zero rows, since the number of possible submatrices consisting of only those $k$ rows is $V(k,n)$.
\end{proof}

We can now conclude the first part of Theorem \ref{thm:largeasymptotics}.

\begin{corollary}\label{cor:bigmAsymptotics}
	We have, for $m > \frac{n^2+n}{2}$,
	\[\max\paren{\binom{m}{\frac{n^2+n}{2}}V\paren{\frac{n^2+n}{2},n},\ \abs{\D_n^{\frac{n^2+n}{2}}}} \leq \abs{\D_n^m} \leq \binom{m}{\frac{n^2+n}{2}}\abs{\D_n^{\frac{n^2+n}{2}}}.\] Hence, if $m(t),n(t)\fromto{\ZZ_{>0}}{\ZZ_{>0}}$ satisfy $m(t),n(t) \to \infty$, then:
	\begin{enumerate}
		\item[(a)] if $m = \Theta(n^2)$ and $m>\frac{n^2+n}{2}$, then $\log \abs{\D_n^m} \sim \log \abs{\D_n^{\frac{n^2+n}{2}}}$;
		\item[(b)] if $\log m = \omega(\log n)$, then $\log \abs{\D_n^m} \sim \log \paren{\dbinom{m}{\frac{n^2+n}{2}}} $.
	\end{enumerate}
\end{corollary}
\begin{proof}
	The lower bound of $\binom{m}{\frac{n^2+n}{2}}V\paren{\frac{n^2+n}{2}}$ follows by taking the last term in Proposition \ref{prop:polynom}, and the lower bound of $\abs{\D_n^{\frac{n^2+n}{2}}}$ follows from noting that $\abs{\D_n^m}$ is an increasing function of $m$. The upper bound follows from applying the inequality \[\dbinom{m}{\frac{n^2+n}{2}} \cdot \dbinom{\frac{n^2+n}{2}}{k} \geq \dbinom{m}{k}\]
	after expanding both sides by Proposition \ref{prop:polynom}. The asymptotic estimates follow directly from the inequalities.
\end{proof}

Proposition \ref{prop:polynom} also tells us that for fixed $n$, $\abs{\D_n^m}$ is a polynomial in $m$ of degree $\frac{n^2+n}{2}$. We now compute the value of $V\paren{\frac{n^2+n}{2}-t,n}$ for $t \in \{0,1\}$, giving us the leading terms of this polynomial. We will first need the following strengthening of Proposition \ref{prop:limited}.

\begin{proposition}\label{prop:strongerlimited}
	Let $\Lambda=(\lambda_1,\ldots,\lambda_a)$ be affinely algebraically independent, and let $M$ be an $a \times (n+1)$ \lek\ matrix. Note that $M$ is not necessarily saturated.
	
	\
	
	Let $T$ be a positive integer such that at most one of the rows of $M$ has more than $T$ ones. Then there is a degree-$n$ polynomial $P$ with the property that $D_P(\Lambda)=M$, and the property that for any $\lambda \in \CC \setminus \Lambda$, at most $T$ of the entries of $D_P((\lambda))$ are one.
\end{proposition}

\begin{proof}
	Append rows of the form $(1,0,\ldots,0)$ to $M$ until it is saturated. Let $M'$ denote the resulting matrix, and let $a'$ be the number of rows of $M'$. Note that $M'$ is \lek. Pick complex numbers $\lambda_{a+1},\ldots,\lambda_{a'}$ such that $\Lambda'=(\lambda_1,\ldots,\lambda_{a'})$ is algebraically independent.
	
	\	
	
	Now, we apply the argument of Proposition \ref{prop:limited} to $M'$ and $\Lambda'$. Consider some $\lambda \in \CC \setminus \Lambda$. We claim that there exist at least two $b \in [a']$ such that the $a'$-tuple $\Lambda_b'$ obtained from $\Lambda'$ by replacing $\lambda_b$ with $\lambda$ is affinely algebraically independent. If $\Lambda'$ with $\lambda$ appended is already affinely algebraically independent, the result is clear.
	
	\
	
	Otherwise, let $Q_0$ be a minimal-degree affine algebraic dependence of $(\lambda,\lambda_1,\ldots,\lambda_{a'})$. As in the proof of Proposition \ref{prop:limited}, $Q_0$ must divide all other affine algebraic dependences. We claim that there are at least $3$ variables with a nonzero coefficient in $Q_0$.
	
	\
	
	Suppose otherwise. Then, there is some nonzero two-variable polynomial $Q$ and distinct constants $\alpha_1,\alpha_2 \in \CC$ such that $Q(a+b\alpha_1,a+b\alpha_2)=0$ for all $a,b \in \CC$. However, substituting $a=\frac{-\alpha_2 x + \alpha_1 y}{\alpha_1-\alpha_2}$ and $b=\frac{x-y}{\alpha_1-\alpha_2}$ implies $Q(x,y) = 0$ for all $x,y \in \CC$, so $Q$ is identically zero, which is a contradiction, as desired.
	
	\
	
	Now, there are at least two values of $b$ such that $x_b$ has nonzero coefficient in $Q_0$. For these values of $b$, we have that $\Lambda'_b$ is affinely algebraically independent. At least one of these choices of $b$ must have the property that the $b$th column of $M'$ has at most $T$ nonzero entries, implying the desired result.
\end{proof}

\begin{proposition}\label{prop:vvalues}
	We have the following values of $V(m,n)$:
	\begin{align*}
		V\paren{\frac{n^2+n}{2},n} &= \frac{\paren{\frac{n^2+n}{2}}!}{1!2!\cdots n!}\quad\text{and}\quad\\
		V\paren{\frac{n^2+n}{2},n-1} &= \frac{\paren{\frac{n^2+n}{2}}!}{1!2!\cdots n!}\paren{1+\frac{(n-1)(n-2)}{4}}.\\
	\end{align*}
\end{proposition}
\begin{proof}
	For the first formula, note that if an element of $\D_n^{\frac{n^2+n}{2}}$ has no all-zero rows, then each row must have exactly one nonzero entry and the $j$th column must have exactly $n+1-j$ ones. We show that all matrices satisfying the aforementioned conditions are in $\D_n^{\frac{n^2+n}{2}}$. This gives the desired enumeration, as we are then placing $n+1-j$ indistinguishable rows for each $j \in [n]$.
	
	\
	
	To show that all such matrices are attainable, consider the polynomial $(x-\lambda_1)\cdots(x-\lambda_n)$, where $\lambda_1,\ldots,\lambda_n$ are affinely algebraically independent. By Proposition \ref{prop:limited} on $P=(x-\lambda_1)\cdots(x-\lambda_n)$, $\Lambda =  (\lambda_1,\ldots,\lambda_n)$, and $T=1$, all of the roots of the derivatives of the $P^{(j)}$ are distinct, implying the desired result.
	
	\
	
	For the second formula, note that if an element of $\D_n^{\frac{n^2+n}{2}-1}$ has all not-all-zero rows, then it must be one of the following:
	\begin{enumerate}
		\item[(a)] an element of $\D_n^{\frac{n^2+n}{2}}$ with no all-zero rows, and with the bottom row removed;
		\item[(b)] a matrix where the $j$th column has exactly $n+1-j$ ones, some row has a nonzero entry in exactly two columns (say $n+1-j$ and $n+1-j'$ with $j'>j+1$), and the remaining rows have exactly one nonzero entry each.
	\end{enumerate}
	The $j'>j$ condition can be assumed by symmetry; the $j' \neq j+1$ condition arises since $P^{(n-j')}(t)=P^{(n-j'-1)}(t)=0$ would imply that $t$ is a multiplicity $2$ root of $P^{(n-j')}$, giving $P^{(n-j')}$ more that $j'$ roots counted with multiplicity.
	
	\
	
	We show that such matrices are attainable. The result is clear for matrices in (a). For matrices in (b), applying Proposition \ref{prop:strongerlimited} to the value $T=1$, the $1$-tuple $\Lambda=(0)$, and the one-row matrix with ones in columns $n+1-j$ and $n+1-j'$, we get a suitable polynomial.
	
	\
	
	Now, we count this set of matrices. There are $V\paren{\frac{n^2+n}{2},n}$ matrices in (a), since the bottom row is uniquely determined by the other rows. The number of matrices in (b) is, conditioning on $j$ and $j'$,

	\begin{align*}
	\sum_{2 \leq j+1 < j'\leq n }\frac{\paren{\frac{n^2+n}{2}-1}!}{1!2!\cdots n!} \cdot jj'&=\frac{\paren{\frac{n^2+n}{2}-1}!}{1!2!\cdots n!} \cdot \sum_{j'=3}^nj' \sum_{j < j'-1}j\\
	&=\frac{\paren{\frac{n^2+n}{2}-1}!}{1!2!\cdots n!} \cdot \sum_{j'=3}^n\frac{j'(j'-1)(j'-2)}{2}\\
	&=\frac{\paren{\frac{n^2+n}{2}-1}!}{1!2!\cdots n!} \cdot \frac{(n+1)n(n-1)(n-2)}{8},\\
	\end{align*}
	giving the desired formula.
\end{proof}

The second part of Theorem \ref{thm:largeasymptotics} directly follows.

\begin{corollary}\label{cor:fixednAsymptotics}
	When $n$ is fixed, we have
	\[\abs{\D_n^m} = \frac{\paren{\frac{n^2+n}{2}}!}{1!2!\cdots n!}\binom{m}{\frac{n^2+n}{2}}+\frac{\paren{\frac{n^2+n}{2}}!}{1!2!\cdots n!}\paren{1+\frac{(n-1)(n-2)}{4}}\binom{m}{\frac{n^2+n}{2}-1}+O\paren{m^{\frac{n^2+n}{2}-2}}.\]
\end{corollary}

\section{Future Work}

In Theorem \ref{thm:asymptotics}, we believe the lower bound is the correct asymptotic formula for $\abs{\D_n^m}$.

\begin{conjecture}
	For $m(t),n(t)\fromto{\ZZ_{>0}}{\ZZ_{>0}}$ satisfying $n(t) \to \infty$ and $1<m(t) \leq \frac{n(t)^2+n(t)}{2}$, we have \[ \log \abs{\D_n^m} \sim \log \paren{n^m\binom{mn}{n}}.\]
\end{conjecture}

If this conjecture is true, then if $m = \omega(n^2)$, $\log \abs{\D_n^m} \sim \log \paren{\dbinom{m}{\frac{n^2+n}{2}} \cdot \abs{\D_n^{\frac{n^2+n}{2}}}}$ follows from the inequalities in Corollary \ref{cor:bigmAsymptotics}, since Stirling's approximation on the formula for $V\paren{\frac{n^2+n}{2},n}$ in Proposition \ref{prop:vvalues} would give $\log\paren{ V\paren{\frac{n^2+n}{2},n}} \sim  \log \abs{\D_n^{\frac{n^2+n}{2}}}$.

\

Our proof of upper bounds on $\D_n^m$ for $m \leq \frac{n^2+n}{2}$ gives no information about the elements of the set $\D_n^m$. A natural further direction to investigate is the properties of the matrices in $\D_n^m$.

\begin{question}
	In terms of $m$ and $n$, how many entries can be ones?
\end{question}

\section{Acknowledgments}

This research was conducted at the University of Minnesota Duluth REU, and was supported by Jane Street Capital, the NSA (grant number H98230-22-1-0015), and the CYAN Undergraduate Mathematics Fund at MIT.

\

We would like to thank Noah Kravitz for his helpful discussions and guidance throughout the research and writing process. We are also grateful to Brian Lawrence for valuable feedback on this paper, and Sean Li for helpful discussions. Finally, we would like to especially thank Joe Gallian for organizing and running the REU.

\printbibliography

\

\textsc{Department of Mathematics, Massachusetts Institute of Technology, Cambridge, MA
02139}

\textit{Email Address:} \href{mailto:ankitb12@mit.edu}{\texttt{ankitb12@mit.edu}}

\end{document}